\documentclass[11pt]{article}
\usepackage[letterpaper,margin=1in]{geometry}
\usepackage{amsmath,amsthm,amssymb,mathtools}
\usepackage{newtxtext,newtxmath}
\usepackage{microtype}
\usepackage{graphicx}
\usepackage{booktabs,tabularx,array}
\usepackage{enumitem}
\usepackage{xcolor}
\usepackage{caption,subcaption}
\usepackage{placeins}
\usepackage{float}
\usepackage[numbers,sort&compress]{natbib}
\usepackage{hyperref}
\usepackage[nameinlink,capitalise,noabbrev]{cleveref}
\usepackage{url}

\definecolor{linkblue}{RGB}{0,78,128}
\hypersetup{
  colorlinks=true,
  linkcolor=linkblue,
  citecolor=linkblue,
  urlcolor=linkblue,
  pdfauthor={Mojtaba Soltanalian and Ahmad Mousavi},
  pdftitle={The Rank-Collapse Principle for Quadratic Optimization}
}
\setlist[itemize]{leftmargin=1.6em,itemsep=0.25em,topsep=0.35em}
\setlist[enumerate]{leftmargin=1.8em,itemsep=0.3em,topsep=0.4em}
\newtheorem{theorem}{Theorem}[section]
\newtheorem{lemma}[theorem]{Lemma}
\newtheorem{proposition}[theorem]{Proposition}
\newtheorem{corollary}[theorem]{Corollary}
\theoremstyle{definition}
\newtheorem{definition}[theorem]{Definition}

\theoremstyle{remark}
\newtheorem{remark}[theorem]{Remark}

\newcommand{\R}{\mathbb{R}}
\newcommand{\C}{\mathbb{C}}
\newcommand{\X}{\mathcal{X}}
\newcommand{\Y}{\mathcal{Y}}

\newcommand{\Sc}{\operatorname{sc}}
\newcommand{\Argmax}{\operatorname*{Arg\,max}}

\newcommand{\conv}{\operatorname{conv}}
\newcommand{\Vtx}{\operatorname{vert}}
\newcommand{\lin}{\operatorname{lin}}
\newcommand{\proj}{\operatorname{proj}}
\newcommand{\rank}{\operatorname{rank}}
\newcommand{\supp}{\operatorname{supp}}

\newcommand{\cO}{\mathcal{O}}

\newcommand{\cD}{\mathcal{D}}
\newcommand{\cH}{\mathcal{H}}

\newcommand{\inner}[2]{\langle #1,#2\rangle}
\newcommand{\norm}[1]{\lVert #1\rVert}
\newcommand{\abs}[1]{\lvert #1\rvert}

\title{The Rank-Collapse Principle for Quadratic Optimization}
\author{
Mojtaba Soltanalian\thanks{Department of Electrical and Computer Engineering, University of Illinois Chicago, Chicago, IL 60607, USA. Email: \href{mailto:msol@uic.edu}{\nolinkurl{msol@uic.edu}}.}
\and
Ahmad Mousavi\thanks{Department of Mathematics and Statistics, American University, Washington, DC 20016, USA. Email: \href{mailto:mousavi@american.edu}{\nolinkurl{mousavi@american.edu}}.}
}
\date{}

\begin{document}

\maketitle

\begin{abstract}
Quadratic optimization becomes hard as soon as either the matrix in the quadratic form has an unfavorable curvature or the feasible set is discrete, combinatorial, or otherwise nonconvex.  A complementary phenomenon is also well known in the signal-processing and optimization communities: when the matrix in the quadratic form has small rank, some hard-looking quadratic programs admit exact polynomial-time algorithms for fixed rank.  We study the common positive-semidefinite geometry behind this phenomenon.  If $Q=BB^{\mathsf T}\succeq0$, the objective depends on $x$ only through the rank-space shadow $y=B^{\mathsf T}x$.  Every optimal shadow $y^\star$ uniquely maximizes the linear functional $y\mapsto\inner{y^\star}{y}$ and satisfies the quantitative margin
\[
\inner{y^\star}{y^\star-y}\ge \tfrac12\norm{y^\star-y}_2^2.
\]
Thus nonlinear optimality collapses to a low-dimensional, self-generated linear exposure direction.  We call this the \emph{rank-collapse principle}.  The principle alone does not imply a finite candidate set: efficient exact optimization additionally depends on the projected or active geometry of the feasible family.  We organize this distinction through projected-shadow scattering and active-structure collapse, relate it explicitly to established zonotope, convex-combinatorial, edge-skeleton, projected-normal-fan, and fixed-rank sparse-PCA methods, and derive tie-safe consequences for binary and finite-phase vectors, cardinality constraints, matroid bases, and sparse PCA.  We also give directional-stability and approximately low-rank certificates, together with reproducible experiments.  The paper's contribution is a unified, careful framework and a set of quantitative consequences, rather than a claim to originate the known fixed-rank tractability results that motivate it.
\end{abstract}


\section{Introduction}
\label{sec:introduction}

Quadratic optimization is difficult for two independent reasons: unfavorable curvature and complicated feasible geometry.  Indefinite quadratic programming is NP-hard under elementary constraints \citep{murty1987quadratic,pardalos1991oneeig}; binary quadratic models include spin glasses and maximum cut \citep{barahona1982ising,goemans1995maxcut}; and semidefinite relaxations are often used precisely because exact optimization is unavailable at scale \citep{nesterov1998semidefinite,luo2010sdr}.  A complementary phenomenon is also well known in the community: when the matrix in the quadratic form has small rank, some hard-looking quadratic programs admit exact polynomial-time algorithms for fixed rank.  Binary, finite-phase, matroid, and sparse-PCA instances provide important examples.

This paper identifies the common geometric statement behind that regime, separates it from the extra assumptions needed for an algorithm, and develops the resulting theory in a form that survives ties, duplicate projections, continuous feasible sets, and approximate low rank.  Consider
\begin{equation}
\label{eq:main-problem}
        \max_{x\in\X} q_Q(x):=x^{\mathsf T}Qx,
        \qquad Q=BB^{\mathsf T}\succeq0,
        \qquad B\in\R^{n\times r},
\end{equation}
where $\X\subseteq\R^n$ is nonempty and compact.  We take $B$ to have full column rank after deleting redundant columns, so $r=\rank(Q)$; later bounds use the possibly smaller effective shadow dimension.  The objective factors as
\begin{equation}
\label{eq:shadow-factor}
        q_Q(x)=\norm{B^{\mathsf T}x}_2^2.
\end{equation}
Thus $Q$ cannot distinguish two feasible points with the same \emph{shadow} $B^{\mathsf T}x$.  The ambient search may be enormous, but the nonlinear objective lives in at most $r$ dimensions.

Low dimensionality alone, however, does not imply a finite oracle image.  Sparse principal component analysis is an instructive example: the active support is piecewise constant in rank space, while the optimal point on a fixed support generally moves continuously.  Conversely, for a discrete polyhedral family, low-dimensional projected normal fans can have only polynomially many cells at fixed rank.  The right theory therefore has three levels:
\begin{enumerate}[label=(\roman*)]
\item \emph{rank-space collapse}: the objective depends only on $B^{\mathsf T}x$;
\item \emph{exposure collapse}: an optimal shadow exposes itself through a rank-space linear functional;
\item \emph{candidate or active-structure collapse}: the feasible-set geometry makes the relevant exposed objects enumerable.
\end{enumerate}
We use \emph{rank-collapse principle} for the first two universal facts and treat the third as the algorithmic consequence that must be proved family by family.

\begin{figure}[t]
\centering
\includegraphics[width=0.92\linewidth]{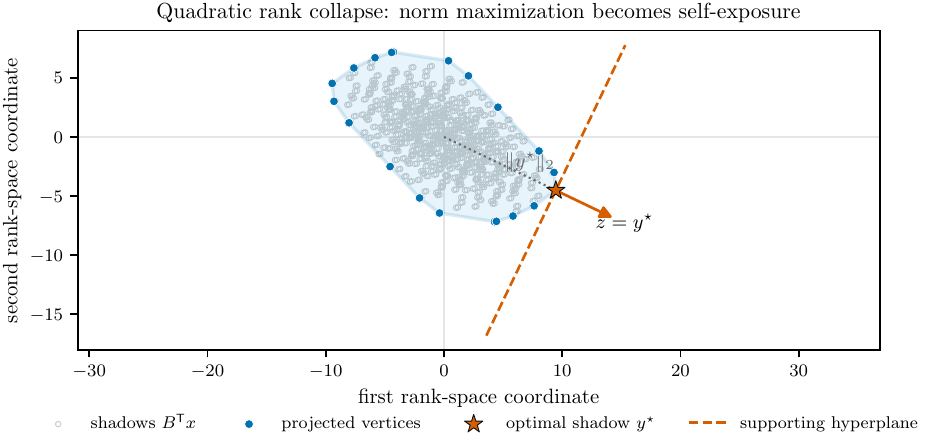}
\caption{The rank-collapse principle in two-dimensional rank space.  The quadratic problem maximizes distance from the origin over the shadow set.  An optimal shadow $y^\star$ is uniquely exposed by the direction $z=y^\star$; the dashed supporting line meets the projected hull only at $y^\star$.  Feasible points sharing this shadow are indistinguishable to the quadratic objective.}
\label{fig:principle}
\end{figure}

\paragraph{Contributions.}
The paper contributes the following organizing results and consequences.  Several associated tractability results are established in prior work; our purpose is to place them in one tie-safe framework and make the additional assumptions visible.
\begin{enumerate}
\item We prove a \emph{self-exposure theorem} with a quadratic separation margin.  It replaces the false statement that an arbitrary deterministic tie-breaking oracle must return every optimizer.  Instead, every oracle response in the self-generated direction has the same optimal shadow and is globally optimal.
\item We define the intrinsic \emph{rank-space scattering number} $\abs{\Vtx(\conv(B^{\mathsf T}\X))}$.  Unlike a count of tie-broken feasible representatives, it is invariant under duplicate shadows, oracle conventions, and redundant descriptions.
\item We give a self-contained edge-direction projected-fan formulation: one feasible-set linear-oracle call per chamber of a central arrangement enumerates all projected vertices.  This is an expository bridge to established zonotope, convex-combinatorial, edge-skeleton, and projected-normal-fan machinery, not a claim that the enumeration mechanism itself is new.
\item We derive structured consequences for binary and finite-phase vectors, top-$k$ subsets, and matroid bases.  For generic projected cubes, the bound is sharp.  In real rank two, a one-pass angular sweep solves binary instances in $O(n\log n)$ time and $O(n)$ memory.
\item We formalize \emph{active-structure collapse}.  A finite collection of supports or faces can suffice even when the point-oracle image is infinite.  This yields a clean exact sparse-PCA derivation and explains the distinction between finite support enumeration and continuous within-support responses.
\item We prove directional perturbation bounds and approximately low-rank objective certificates.  For a positive- or negative-semidefinite residual, the additive loss improves from the universal $2M^2\norm{E}_2$ to $M^2\norm{E}_2$, and a computable upper bound enables rank-adaptive certification.
\end{enumerate}

\paragraph{Scope and limits.}
The universal theorem concerns maximization of a positive-semidefinite quadratic.  It does not assert that every low-rank quadratic problem is easy.  Indeed, rank-one negative-semidefinite maximization over signs already encodes \textsc{Partition}; see \cref{sec:limits}.  Nor does low rank alone make the candidate set finite.  What low rank guarantees is a low-dimensional exposure direction.  Polynomial solvability requires a finite or otherwise manageable projected fan or active-structure family.

\paragraph{Organization.}
\Cref{sec:literature} positions the contribution.  \Cref{sec:principle} proves self-exposure, stability, and the convex extension.  \Cref{sec:scattering} develops scattering and projected-fan enumeration.  \Cref{sec:families} treats structured discrete families, and \cref{sec:active} develops active-structure collapse and sparse PCA.  \Cref{sec:approximate} gives approximate-rank certificates, \cref{sec:limits} gives the curvature boundary, and \cref{sec:experiments} reports computational evidence.  Complete degeneracy arguments, complex realification, an oracle-ascent result, and full experimental protocols appear in the appendices and are explicitly cited below.

\section{Relation to prior work}
\label{sec:literature}

\paragraph{Fixed-rank binary and phase optimization.}
The closest class-specific predecessor is the zonotope algorithm of Ferrez, Fukuda, and Liebling \citep{ferrez2005zonotope} for fixed-rank convex quadratic maximization in binary variables.  Karystinos and Liavas \citep{karystinos2010binary} later developed an efficient hyperspherical partition method, and Kyrillidis and Karystinos \citep{kyrillidis2014mpsk} treated finite-phase sequences.  Those papers already establish fixed-rank tractability for important alphabets.  Our aim is not to relabel those algorithms as new.  We isolate a tie-safe quadratic self-exposure identity, define an intrinsic projected complexity measure, and show how the same logical structure extends to oracle-defined combinatorial families, active supports, directional stability, and approximate-rank certificates.

\paragraph{Convex combinatorial optimization.}
Onn and Rothblum \citep{onnrothblum2004convex} showed how fixed-dimensional convex optimization over combinatorial families can reduce to linear optimization using projected polyhedral geometry; Onn's convex matroid optimization result \citep{onn2003matroid} is a particularly relevant specialization, and \citet{onn2010nonlinear} gives a broader treatment.  This literature is the natural home of the general convex extension in \cref{thm:convex-exposure}.  The quadratic theorem is stronger than generic convex exposure: its subgradient is the optimal shadow itself, and the resulting identity yields uniqueness and a Euclidean stability margin.

\paragraph{Projected normal fans and low-dimensional objectives.}
Emiris, Fisikopoulos, and G{\"a}rtner \citep{emiris2016edgeskeleton} give a total-polynomial edge-skeleton and vertex-enumeration framework for polytopes specified by optimization or separation oracles together with a superset of edge directions.  Scott and Geunes \citep{scottgeunes2025normalfan} give a projected normal-fan algorithm for low-rank nonlinear optimization over polytopes.  Edge directions and normal-fan refinement also underlie the convex-combinatorial framework.  Our projected-fan theorem should be read as a self-contained bridge from quadratic self-exposure to that machinery, not as a claim that hyperplane arrangements are new.  Related low-dimensional objective algorithms and approximation schemes include \citet{mittal2013lowrank}, \citet{hunkenschroder2023cube}, and \citet{dadush2023lowdim}; they address different oracle, integrality, minimization, or approximation regimes.  Recent factorized binary polynomial optimization \citep{delpia2025factorized} further demonstrates that fixed factorization dimension can reveal strongly polynomial discrete algorithms.

\paragraph{Sparse PCA.}
Constant-rank sparse PCA was treated by Asteris, Papailiopoulos, and Karystinos \citep{asteris2014spca}, and Del Pia \citep{delpia2023spca} subsequently established polynomial-time global optimization for sparse PCA at fixed covariance-matrix rank, including a disjoint-support variant.  Our sparse-PCA section gives a different conceptual contribution: it uses the example to show why finite point-candidate theory is too narrow and proves a finite \emph{active-structure} theorem that remains valid when the point oracle varies continuously within each support cell.

\paragraph{What is contributed here.}
The paper packages the following chain in a single notation and makes clear which arrows are universal and which require family-specific geometry:
\begin{equation*}
\begin{aligned}
\text{PSD quadratic optimum}
&\Longrightarrow \text{unique self-exposed shadow with margin}\\
&\Longrightarrow \text{projected-vertex or active-structure enumeration}\\
&\Longrightarrow \text{exact or certified optimization}.
\end{aligned}
\end{equation*}
The first implication and its quantitative consequences are quadratic-specific; the second makes the algorithmic assumptions explicit; and the approximate-rank theory shows how the exact principle degrades under spectral residuals.  This formulation avoids conflating low rank, finite oracle images, and polynomial-time enumeration.

\section{The rank-collapse principle}
\label{sec:principle}

Throughout this section, $\X\subseteq\R^n$ is nonempty and compact, $B\in\R^{n\times r}$, and $Q=BB^{\mathsf T}$.  Define the shadow set and its compact convex hull by
\begin{equation}
\label{eq:Y-P}
       \Y_B:=B^{\mathsf T}\X,
       \qquad
       P_B:=\conv(\Y_B)\subseteq\R^r.
\end{equation}
Maximizing $\norm{y}_2^2$ over $\Y_B$ or $P_B$ gives the same value: convexity bounds the value at a convex combination by the largest value among its constituents.

For $z\in\R^r$, define the set-valued rank-space linear oracle
\begin{equation}
\label{eq:oracle}
       \cO_B(z):=\Argmax_{x\in\X}\inner{z}{B^{\mathsf T}x}.
\end{equation}
The set-valued definition is essential.  A deterministic selection from an argmax set can discard an optimizer that shares its shadow with another feasible point.  The zero direction is also special: $\cO_B(0)=\X$, so it contains no information.  Meaningful exposure uses $z\ne0$.

\begin{definition}[shadow equivalence]
Two feasible points are shadow-equivalent, written $x\sim_B x'$, if $B^{\mathsf T}x=B^{\mathsf T}x'$.  The quadratic objective is constant on each equivalence class.
\end{definition}

\subsection{Self-exposure and the quadratic margin}

\begin{theorem}[Rank-collapse principle: unique self-exposure]
\label{thm:self-exposure}
Let $x^\star$ solve \eqref{eq:main-problem} and set $y^\star=B^{\mathsf T}x^\star$.  Then, for every $y\in P_B$,
\begin{equation}
\label{eq:margin-identity}
\inner{y^\star}{y^\star-y}
 =\frac12\Bigl(\norm{y^\star}_2^2-\norm{y}_2^2+\norm{y^\star-y}_2^2\Bigr)
 \ge \frac12\norm{y^\star-y}_2^2.
\end{equation}
Consequently, $y^\star$ is the unique maximizer of $y\mapsto\inner{y^\star}{y}$ over $P_B$.  Moreover,
\begin{equation}
\label{eq:oracle-self}
      x^\star\in\cO_B(y^\star),
      \qquad
      x\in\cO_B(y^\star)\ \Longrightarrow\ B^{\mathsf T}x=y^\star
      \ \Longrightarrow\ x\text{ is globally optimal.}
\end{equation}
\end{theorem}

\begin{proof}
Optimality of $y^\star$ for norm maximization gives $\norm{y}_2^2\le\norm{y^\star}_2^2$ for every $y\in P_B$.  The polarization identity yields the equality in \eqref{eq:margin-identity}; dropping its nonnegative first difference gives the inequality.  If $y\ne y^\star$, the right-hand side is positive, so $\inner{y^\star}{y}<\norm{y^\star}_2^2$ and the linear maximizer is unique in shadow space.  Every $x\in\cO_B(y^\star)$ therefore has shadow $y^\star$ and hence objective $\norm{y^\star}_2^2$.
\end{proof}

\begin{remark}[The zero-optimum case]
If $\norm{y^\star}_2=0$, then every shadow in $P_B$ is zero, every feasible point is optimal, and no nonzero exposing direction can distinguish feasible points.  Otherwise $y^\star\ne0$, so \cref{thm:self-exposure} supplies a nontrivial rank-space direction.
\end{remark}

The theorem is stronger than the familiar supporting-linearization inequality.  It identifies the exposing normal exactly and gives a quadratic separation margin, which is the basis of the perturbation result below.

\subsection{Directional stability}

\begin{theorem}[Stability of an inexact exposing direction]
\label{thm:direction-stability}
Let $y^\star$ be an optimal shadow, let $\widehat z\in\R^r$ satisfy $\norm{\widehat z-y^\star}_2\le\varepsilon$, and let
\[
       \widehat y\in\Argmax_{y\in P_B}\inner{\widehat z}{y}.
\]
Then
\begin{equation}
\label{eq:direction-stability}
       \norm{\widehat y-y^\star}_2\le2\varepsilon.
\end{equation}
If $R:=\norm{y^\star}_2$, then
\begin{equation}
\label{eq:direction-objective}
       0\le \norm{y^\star}_2^2-\norm{\widehat y}_2^2\le4R\varepsilon.
\end{equation}
If $\Y_B=\{y^\star\}$, exact recovery is immediate.  Otherwise, if $\Y_B$ is finite and
\[
       \delta:=\min\{\norm{y-y^\star}_2:y\in\Y_B,\ y\ne y^\star\}>0,
\]
then $\varepsilon<\delta/2$ implies $\widehat y=y^\star$.
\end{theorem}

\begin{proof}
Let $d=y^\star-\widehat y$.  By \eqref{eq:margin-identity}, $\tfrac12\norm d_2^2\le\inner{y^\star}{d}$.  Since $\widehat y$ maximizes the perturbed linear functional, $\inner{\widehat z}{d}\le0$.  Hence
\[
\frac12\norm d_2^2
\le \inner{y^\star-\widehat z}{d}+\inner{\widehat z}{d}
\le \varepsilon\norm d_2,
\]
which proves \eqref{eq:direction-stability}.  Because $y^\star$ has maximum norm, $\norm{\widehat y}_2\le R$, and therefore
\[
\norm{y^\star}_2^2-\norm{\widehat y}_2^2
=\inner{y^\star+\widehat y}{d}
\le2R\norm d_2\le4R\varepsilon.
\]
In the finite case, fix any $y\in\Y_B\setminus\{y^\star\}$ and write $d_y=y^\star-y$.  Then
\[
\inner{\widehat z}{d_y}
=\inner{y^\star}{d_y}+\inner{\widehat z-y^\star}{d_y}
\ge \tfrac12\norm{d_y}_2^2-\varepsilon\norm{d_y}_2>0
\]
when $\varepsilon<\delta/2$.  Hence $y^\star$ strictly beats every other element of $\Y_B$ under the perturbed functional and therefore uniquely maximizes it over $P_B=\conv(\Y_B)$.
\end{proof}

\Cref{thm:direction-stability} is independent of a particular enumeration algorithm.  It quantifies how accurately an exposing direction must be computed, and it yields an exact-recovery radius whenever the relevant shadow set is separated.  Appendix~\ref{secS:normal} gives an equivalent normal-cone interpretation.

\subsection{A convex low-dimensional extension}

The exposure argument extends beyond squared norms, but the quadratic uniqueness and margin do not.

\begin{theorem}[Convex exposure]
\label{thm:convex-exposure}
Let $f:\R^r\to\R$ be finite and convex, and let $x^\star$ maximize $f(B^{\mathsf T}x)$ over $\X$.  Put $y^\star=B^{\mathsf T}x^\star$.  For every $g\in\partial f(y^\star)$,
\begin{equation}
\label{eq:convex-exposure}
        x^\star\in\cO_B(g).
\end{equation}
If $g=0$, the conclusion is formally true but algorithmically vacuous.
\end{theorem}

\begin{proof}
For any feasible shadow $y$, convexity gives
$f(y)\ge f(y^\star)+\inner{g}{y-y^\star}$, while global maximality gives $f(y)\le f(y^\star)$.  Thus $\inner{g}{y-y^\star}\le0$, which is exactly \eqref{eq:convex-exposure}.
\end{proof}

This theorem is consistent with convex combinatorial optimization \citep{onnrothblum2004convex,onn2010nonlinear}.  For $f(y)=\norm y_2^2$, one may take $g=2y^\star$; \cref{thm:self-exposure} then adds uniqueness, the exact self-generated normal, and the margin \eqref{eq:margin-identity}.

\section{Intrinsic scattering and projected-fan enumeration}
\label{sec:scattering}

Assume in this section that $P:=\conv(\X)$ is a polytope.  Linear optimization over $\X$ and $P$ has the same value, and the projected hull is $P_B=B^{\mathsf T}P$.

\begin{definition}[rank-space scattering]
\label{def:scattering}
The scattering number of $(P,B)$ is
\begin{equation}
\label{eq:scattering}
        \Sc_B(P):=\abs{\Vtx(P_B)}.
\end{equation}
For a family $\mathfrak P_n$, its rank-$r$ scattering envelope is
\[
        \Sc_{\mathfrak P}(n,r)
        :=\sup\{\Sc_B(P):P\in\mathfrak P_n,\ \rank(B)\le r\}.
\]
\end{definition}

This definition counts distinct exposed shadows, not feasible representatives.  It is invariant under duplicated feasible points, noninjective projections, and oracle tie-breaking.  By \cref{thm:self-exposure}, every optimal shadow is an exposed point of $P_B$ and hence a vertex.  Therefore enumerating one preimage of every projected vertex solves \eqref{eq:main-problem} exactly.

\subsection{Enumeration from edge directions}

A finite set $\cD\subset\R^n\setminus\{0\}$ \emph{covers the edge directions} of $P$ if every edge vector of $P$ is a nonzero scalar multiple of some $d\in\cD$.  Set
\[
       L:=\lin(P_B-P_B),\qquad d:=\dim L,
\]
For $d_0\in\cD$, let
\[
       g(d_0):=\proj_L(B^{\mathsf T}d_0),
\]
where $\proj_L$ is orthogonal projection onto $L$.  Discard the directions for which $g(d_0)=0$, identify nonzero scalar multiples, and form the resulting proper central hyperplanes
\begin{equation}
\label{eq:edge-arrangement}
       \cH_B(\cD):=\{H_g:g=g(d_0)\ne0,\ d_0\in\cD\},
       \qquad H_g:=\{z\in L:\inner{z}{g}=0\}.
\end{equation}
Let $N:=\abs{\cH_B(\cD)}$.  Projecting onto $L$ is essential when $\cD$ is a strict superset of the actual edge directions: an extraneous direction orthogonal to $L$ must not create the improper ``hyperplane'' $L$ itself.

\begin{theorem}[Projected edge-direction enumeration]
\label{thm:edge-enum}
Choose one direction $z_C$ from every full-dimensional chamber $C$ of $L\setminus\bigcup\cH_B(\cD)$, and for each $C$ solve
\[
        x_C\in\Argmax_{x\in\X}\inner{z_C}{B^{\mathsf T}x}.
\]
Then $\{B^{\mathsf T}x_C\}_C$ contains every vertex of $P_B$.  Consequently,
\begin{equation}
\label{eq:central-count}
\Sc_B(P)\le R_{\mathrm c}(N,d):=
\begin{cases}
1, & d=0\text{ or }N=0,\\[0.2em]
\displaystyle 2\sum_{j=0}^{d-1}\binom{N-1}{j}, & d\ge1,\ N\ge1.
\end{cases}
\end{equation}
For fixed $d$, the bound is $O(N^{d-1})$.
\end{theorem}

\begin{proof}
Fix $z$ in a chamber and let $F$ be the maximizing face of $P$.  If $B^{\mathsf T}F$ contained two points, the connected edge graph of $F$ would contain an edge $[u,v]$ with $B^{\mathsf T}(v-u)\ne0$.  Because $u$ and $v$ both maximize the linear functional,
$\inner{z}{B^{\mathsf T}(v-u)}=0$.  The edge direction is covered by some $d_0\in\cD$.  Since $B^{\mathsf T}(v-u)\in L$, it is a nonzero scalar multiple of $g(d_0)$, so $z$ lies on a hyperplane of \eqref{eq:edge-arrangement}, a contradiction.  Thus $B^{\mathsf T}F$ is a singleton and every feasible-set oracle response at $z$ has the same projected vertex.

The projected response is locally constant inside the chamber: at any $z$, its projected maximizer is unique among the finitely many vertices of $P_B$, so the strict inequalities persist in a neighborhood.  Because a chamber is connected, the projected response is constant throughout it.  Conversely, every vertex $y$ of $P_B$ has a full-dimensional relative normal cone in $L$.  The interior of that cone cannot be covered by finitely many proper hyperplanes, so it meets a chamber, whose responses therefore project to $y$.  The chamber bound is the standard maximum number of regions of a central arrangement \citep{zaslavsky1975arrangements,matousek2002discrete}.  Appendix~\ref{secS:edge} records the same argument with the lower-dimensional and degeneracy conventions made explicit.
\end{proof}

\begin{remark}[Computational model]
The theorem separates geometry from implementation.  If a linear oracle over $\X$ is polynomial-time and the arrangement can be traversed or its chambers output in polynomial time per chamber, the algorithm is output-sensitive in $R_{\mathrm c}(N,d)$.  Standard arrangement methods supply such enumeration for fixed $d$ \citep{edelsbrunner1987algorithms,matousek2002discrete}.  Oracle edge-skeleton methods provide a stronger total-polynomial framework in suitable encodings \citep{emiris2016edgeskeleton}.  Projected-normal-fan methods can be preferable when they avoid listing an overly large direction cover \citep{scottgeunes2025normalfan}.  The oracle-call and arrangement bounds here are stated in the real-arithmetic/linear-oracle model; rational bit complexity depends on the representation of chamber directions and on the underlying oracle.
\end{remark}

\begin{corollary}[Exact fixed-dimensional optimization]
\label{cor:fixed-dim}
Under the assumptions of \cref{thm:edge-enum}, evaluate $\norm{B^{\mathsf T}x_C}_2^2$ for the distinct projected responses and return the largest.  This solves \eqref{eq:main-problem} exactly using at most $R_{\mathrm c}(N,d)$ linear-oracle calls.
\end{corollary}

\section{Structured discrete families}
\label{sec:families}

\Cref{tab:families} summarizes the resulting bounds.  The exponent depends on the effective projected dimension $d\le r$ in real problems and $d\le2r$ after complex realification.  The complex derivation and phase-boundary bookkeeping are in Appendix~\ref{secS:complex}.

\begin{table}[t]
\centering
\caption{Representative rank-collapse bounds.  $N$ is the number of distinct projected edge hyperplanes and $d$ is the effective real shadow dimension.  Degenerate instances can only reduce the chamber count.}
\label{tab:families}
\small
\begin{tabularx}{\linewidth}{@{}l l l X@{}}
\toprule
Feasible family & Linear oracle & Direction bound & Scattering / structures \\
\midrule
$\{\pm1\}^n$ & coordinate signs & $N\le n$ & $R_{\mathrm c}(n,d)$; sharp generically \\
$m$-phase vectors & phase quantization & $N\le mn$ & $R_{\mathrm c}(mn,d)$ \\
Top-$k$ subsets & $k$ largest scores & $N\le \binom{n}{2}$ & $O(n^{2d-2})$ \\
Matroid bases & greedy algorithm & $N\le \binom{n}{2}$ & $O(n^{2d-2})$ \\
Sparse PCA supports & $k$ largest $\abs{(Bz)_i}$ & $N\le n(n-1)$ & $O(n^{2r-2})$ active supports \\
\bottomrule
\end{tabularx}
\end{table}

\subsection{Binary vectors and zonotopes}

Let $\X=\{\pm1\}^n$.  Its convex hull is the cube, whose edge directions are the coordinate vectors.  The projected hull is the zonotope
\begin{equation}
\label{eq:zonotope}
       P_B=B^{\mathsf T}[-1,1]^n
       =\sum_{i=1}^n[-b_i,b_i],
\end{equation}
where $b_i^{\mathsf T}$ is row $i$ of $B$.

\begin{corollary}[Binary scattering]
\label{cor:binary}
Let $m:=\abs{\{i:b_i\ne0\}}$ and $d=\dim P_B$.  If $d=0$, then $P_B$ is a singleton and $\Sc_B([-1,1]^n)=1$.  If $d\ge1$, then
\begin{equation}
\label{eq:binary-bound}
       \Sc_B([-1,1]^n)\le2\sum_{j=0}^{d-1}\binom{m-1}{j}
       \le2\sum_{j=0}^{d-1}\binom{n-1}{j}=O(n^{d-1}).
\end{equation}
If the $m$ nonzero generators are in general linear position in $L$, equality holds in the first bound.  Hence binary positive-semidefinite quadratic maximization is polynomial-time solvable for fixed rank.
\end{corollary}

This is the classical zonotope count underlying the algorithm of \citet{ferrez2005zonotope}; the hyperspherical binary algorithm of \citet{karystinos2010binary} realizes the same fixed-rank phenomenon through sign regions.

For $d=2$, exactly $2m$ projected vertices occur in general position.  Because the objective is even, one representative from each antipodal pair is enough.

\begin{proposition}[Rank-two angular sweep]
\label{prop:rank2-sweep}
For $B\in\R^{n\times2}$ with nonzero, pairwise nonparallel rows, binary problem \eqref{eq:main-problem} can be solved in $O(n\log n)$ time and $O(n)$ memory.  After sorting the $n$ event angles at which $b_i^{\mathsf T}z=0$, a sweep across one half-circle flips one sign at each event and updates the shadow in $O(1)$ time.
\end{proposition}

\begin{proof}
Within an open angular cell, the linear oracle is $x_i=\operatorname{sign}(b_i^{\mathsf T}z)$.  Crossing the event for row $i$ changes only $x_i$, so $y=B^{\mathsf T}x$ updates as $y\leftarrow y-2x_i b_i$.  Sorting costs $O(n\log n)$ and the sweep costs $O(n)$.  Every projected vertex has a normal cone containing one open cell by \cref{thm:edge-enum}; antipodal cells have equal quadratic value.  Simultaneous events are handled by grouping or symbolic perturbation, as detailed in Appendix~\ref{secS:degeneracy}.
\end{proof}

\subsection{Finite-phase alphabets}

Let $m\ge3$, $\Omega_m=\{e^{2\pi\mathrm i q/m}:q=0,\ldots,m-1\}$, and $\X=\Omega_m^n$; the case $m=2$ is the binary family above.  Realifying $\C^r$ gives an effective dimension $d\le2r$.  The convex hull is a product of regular $m$-gons; each coordinate contributes at most $m$ edge directions.

\begin{corollary}[Finite-phase scattering]
\label{cor:phase}
For an $m$-phase alphabet,
\begin{equation}
       \Sc_B(\conv(\Omega_m^n))\le R_{\mathrm c}(mn,d)=O((mn)^{d-1}).
\end{equation}
For fixed $m$ and complex rank $r$, this is $O(n^{2r-1})$.  For complex rank one, an angular sweep uses at most $mn$ phase events.
\end{corollary}

The result recovers the polynomial fixed-rank behavior studied in \citet{kyrillidis2014mpsk}, while placing it in the same projected-edge framework as binary vectors.

\subsection{Cardinality constraints and matroids}

Let
\[
      \X_{n,k}=\{x\in\{0,1\}^n:\mathbf1^{\mathsf T}x=k\},
      \qquad 0<k<n.
\]
The linear oracle selects the $k$ largest coordinates of $Bz$, and the hypersimplex $\conv(\X_{n,k})$ has edge directions $e_i-e_j$.

\begin{corollary}[Top-$k$ subsets]
\label{cor:topk}
With $d=\dim P_B$,
\begin{equation}
\Sc_B(\conv(\X_{n,k}))
\le R_{\mathrm c}\!\left(\binom{n}{2},d\right)
=O(n^{2d-2}).
\end{equation}
The cases $k=0,n$ have scattering one.
\end{corollary}

The same direction set governs every matroid base polytope: each edge is parallel to $e_i-e_j$, and a maximum-weight base is obtained by the greedy algorithm \citep{edmonds1971matroids,schrijver2003combinatorial}.

\begin{corollary}[Matroid bases]
\label{cor:matroid}
Let $\X$ be the incidence vectors of the bases of a matroid on $n$ elements.  Then fixed-rank positive-semidefinite quadratic maximization over $\X$ is solvable using at most
$R_{\mathrm c}(\binom{n}{2},d)$ greedy-oracle calls.  In particular, this applies to spanning trees through the graphic matroid.
\end{corollary}

This corollary is consistent with the broader convex matroid framework of \citet{onn2003matroid}; its role here is to show that the rank-collapse principle interfaces directly with a standard combinatorial oracle.

\begin{figure}[t]
\centering
\begin{subfigure}[t]{0.49\linewidth}
\centering
\includegraphics[width=\linewidth]{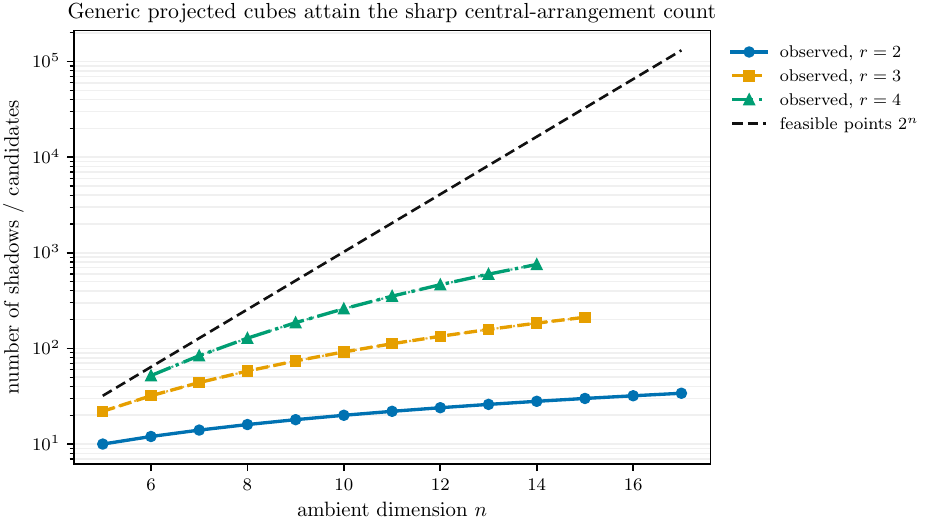}
\caption{Sharp projected-cube counts.}
\label{fig:binary-scattering}
\end{subfigure}\hfill
\begin{subfigure}[t]{0.49\linewidth}
\centering
\includegraphics[width=\linewidth]{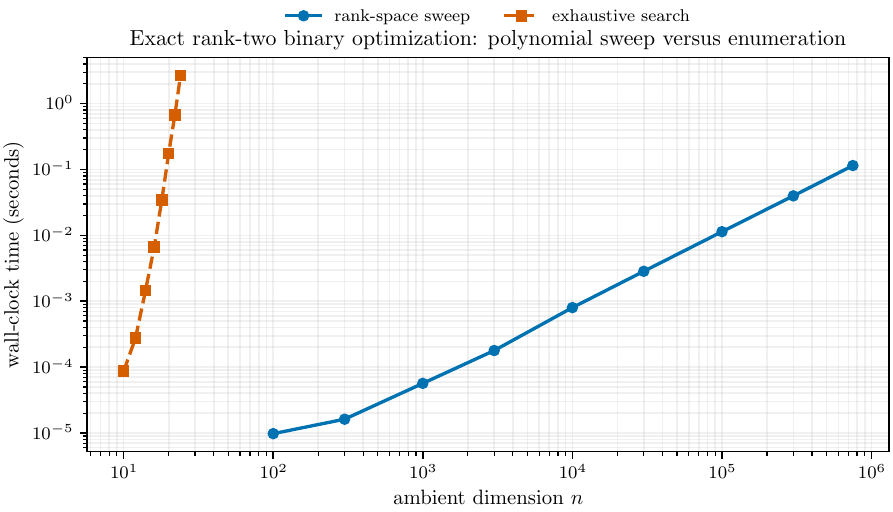}
\caption{Exact rank-two runtime scaling.}
\label{fig:binary-runtime}
\end{subfigure}
\caption{Binary rank collapse.  Left: all 132 generic Gaussian computations match the central-arrangement formula, while the feasible set grows as $2^n$.  Right: the angular sweep remains subsecond through $n=750{,}000$ on the recorded CPU, whereas exhaustive enumeration reaches seconds at $n=24$.  Details are in \cref{sec:experiments}.}
\label{fig:binary-combined}
\end{figure}

\section{Active-structure collapse and sparse PCA}
\label{sec:active}

Finite projected vertices are not the only useful notion of collapse.  Suppose
\begin{equation}
\label{eq:union-structures}
        \X=\bigcup_{\sigma\in\Sigma}\X_\sigma,
\end{equation}
where $\Sigma$ is finite and every $\X_\sigma$ is compact.  Define support functions
\[
        h_\sigma(z):=\max_{x\in\X_\sigma}\inner{z}{B^{\mathsf T}x},
        \qquad h(z):=\max_{\sigma\in\Sigma}h_\sigma(z).
\]
Let $\mathcal R\subset\R^r$ be dense, typically the complement of a finite arrangement, and call $\sigma$ \emph{regularly active} if $h_\sigma(z)=h(z)$ for some $z\in\mathcal R$.

\begin{theorem}[Finite active-structure collapse]
\label{thm:active-structure}
Assume the optimal value of \eqref{eq:main-problem} is positive.  Let $\Sigma_{\mathcal R}$ be the set of regularly active structures.  Then at least one global optimizer belongs to
\[
        \bigcup_{\sigma\in\Sigma_{\mathcal R}}\X_\sigma.
\]
Consequently, if each restricted quadratic problem over $\X_\sigma$ can be solved exactly, solving it for all $\sigma\in\Sigma_{\mathcal R}$ and taking the best result solves the original problem.
\end{theorem}

\begin{proof}
Let $y^\star\ne0$ be an optimal shadow and choose $z_k\in\mathcal R$ with $z_k\to y^\star$.  For each $k$, choose an active structure $\sigma_k$.  Finiteness of $\Sigma$ gives a subsequence on which $\sigma_k=\bar\sigma$.  Support functions of compact sets are continuous, so
$h_{\bar\sigma}(y^\star)=h(y^\star)$.  Choose $\bar x\in\X_{\bar\sigma}$ attaining this value.  Then $\bar x\in\cO_B(y^\star)$; by \cref{thm:self-exposure}, $B^{\mathsf T}\bar x=y^\star$ and $\bar x$ is globally optimal.
\end{proof}

The theorem does not require the point-oracle map to be constant.  Only the identity of a finite active structure must be locally enumerable.

\subsection{Sparse principal component analysis}

Consider the rank-$r$ sparse-PCA problem
\begin{equation}
\label{eq:spca}
       \max_{x\in\R^n}x^{\mathsf T}BB^{\mathsf T}x
       \quad\text{s.t.}\quad
       \norm{x}_2=1,\quad \norm{x}_0\le k.
\end{equation}
For a rank-space direction $z$, write $w=Bz$.  On a fixed support $S$ with $\abs S=k$,
\begin{equation}
\label{eq:spca-linear}
       \max_{\supp(x)\subseteq S,\ \norm{x}_2=1} w^{\mathsf T}x
       =\norm{w_S}_2,
       \qquad
       x_S(z)=\frac{w_S}{\norm{w_S}_2}
\end{equation}
when $w_S\ne0$.  A regular optimal support contains the $k$ largest magnitudes $\abs{w_i}$.  Its identity changes only on
\begin{equation}
\label{eq:spca-events}
       (b_i-b_j)^{\mathsf T}z=0
       \quad\text{or}\quad
       (b_i+b_j)^{\mathsf T}z=0,
       \qquad i<j.
\end{equation}
There are at most $n(n-1)$ distinct central hyperplanes.

\begin{theorem}[Fixed-rank sparse PCA by active supports: generic case]
\label{thm:spca}
Assume $1\le k\le n$, every row $b_i$ is nonzero, and $b_i\ne\pm b_j$ for $i\ne j$.  Let $\mathcal S$ be the supports observed in the full-dimensional cells of arrangement \eqref{eq:spca-events}.  Then
\begin{equation}
\label{eq:spca-enum}
       \operatorname{opt}(\text{\eqref{eq:spca}})
       =\max_{S\in\mathcal S}\lambda_{\max}(B_S^{\mathsf T}B_S).
\end{equation}
Moreover,
\begin{equation}
       \abs{\mathcal S}\le R_{\mathrm c}(n(n-1),r)=O(n^{2r-2})
\end{equation}
for fixed real rank $r$.
\end{theorem}

\begin{proof}
Use structures $\X_S=\{x:\supp(x)\subseteq S,\ \norm{x}_2=1\}$ for all $k$-element sets $S$.  Under the stated nondegeneracy assumptions, outside \eqref{eq:spca-events} all magnitudes are distinct, and the top-$k$ support is unique and constant within each cell.  Thus \cref{thm:active-structure} shows that one enumerated support contains a global optimizer.  The restricted quadratic optimum is the largest eigenvalue of $B_SB_S^{\mathsf T}$, equal to the largest eigenvalue of $B_S^{\mathsf T}B_S$.  The central-arrangement count gives the bound.
\end{proof}

The point-oracle image is nevertheless generally infinite.  On a cell with support $S$, \eqref{eq:spca-linear} varies continuously with $z$.  If it were constant on a full-dimensional open set, linearity would force the range of $B_S$ to lie on one line; hence $\rank(B_S)\le1$.  Thus $\rank(B_S)\ge2$ generically produces infinitely many point responses within a single support cell.  In rank one, by contrast, the support is fixed away from $z=0$ and there are at most two antipodal responses.  This distinction is the reason for introducing active-structure rather than point-candidate collapse.  The fixed-rank tractability is already known from \citet{asteris2014spca,delpia2023spca}; the role of this derivation is to clarify its oracle geometry.  Permanent ties such as $b_i=\pm b_j$ fall outside the displayed generic theorem and require tie classes or a symbolic-perturbation/limit treatment; Appendix~\ref{secS:degeneracy} records the convention used by the enumeration code.  If $B=0$, every feasible vector is optimal.

\begin{figure}[t]
\centering
\includegraphics[width=0.94\linewidth]{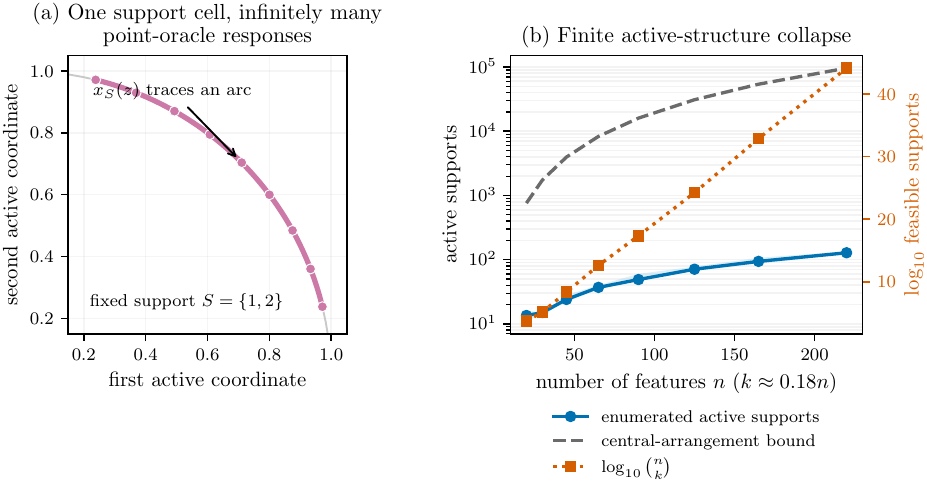}
\caption{Sparse PCA illustrates active-structure collapse.  Left: one support remains fixed while the normalized point-oracle response traces a continuum.  Right: on rank-two Gaussian instances, the number of enumerated active supports remains tiny relative to $\binom{n}{k}$; all 32 independent small-instance comparisons with exhaustive support search were exact.}
\label{fig:spca}
\end{figure}

\FloatBarrier

\section{Approximately low-rank objectives and certification}
\label{sec:approximate}

Exact low rank is rarely literal in data.  Let
\begin{equation}
\label{eq:approx-decomp}
       Q=Q_0+E,
       \qquad Q_0=BB^{\mathsf T}\succeq0,
\end{equation}
where $E=E^{\mathsf T}$ and $\norm{x}_2\le M$ for every $x\in\X$.  Suppose $\varnothing\ne\mathcal C\subseteq\X$ is a finite candidate set containing at least one exact maximizer $x_0$ of $q_{Q_0}$---for example, $\mathcal C$ is obtained from rank-collapse enumeration for $Q_0$---and let
\begin{equation}
\label{eq:approx-choice}
       \widehat x\in\Argmax_{x\in\mathcal C}q_Q(x).
\end{equation}
Define the one-sided spectral width
\begin{equation}
\label{eq:spectral-width}
       \omega(E):=\max\{\lambda_{\max}(E),0\}
                   -\min\{\lambda_{\min}(E),0\}.
\end{equation}

\begin{theorem}[Approximate rank-collapse certificate]
\label{thm:approximate}
Under \eqref{eq:approx-decomp}--\eqref{eq:approx-choice},
\begin{equation}
\label{eq:approx-bound}
       0\le \max_{x\in\X}q_Q(x)-q_Q(\widehat x)
       \le M^2\omega(E)
       \le2M^2\norm{E}_2.
\end{equation}
If $E\succeq0$ or $E\preceq0$, the right-hand side improves to $M^2\norm E_2$.  Writing
\[
       v_0:=\max_{x\in\X}q_{Q_0}(x),
       \qquad
       U:=v_0+M^2\max\{\lambda_{\max}(E),0\},
\]
we also have the computable a posteriori certificate
\begin{equation}
\label{eq:posterior}
       q_Q(\widehat x)\le\operatorname{opt}_Q\le U,
       \qquad
       \operatorname{opt}_Q-q_Q(\widehat x)\le U-q_Q(\widehat x).
\end{equation}
\end{theorem}

\begin{proof}
Let $x^\star$ maximize $q_Q$.  Since $\widehat x$ maximizes $q_Q$ over a set containing $x_0$,
\begin{align*}
q_Q(x^\star)-q_Q(\widehat x)
&\le q_Q(x^\star)-q_Q(x_0)\\
&=q_{Q_0}(x^\star)-q_{Q_0}(x_0)
   +(x^\star)^{\mathsf T}Ex^\star-x_0^{\mathsf T}Ex_0\\
&\le M^2\max\{\lambda_{\max}(E),0\}
   -M^2\min\{\lambda_{\min}(E),0\}.
\end{align*}
This is \eqref{eq:approx-bound}; the semidefinite and norm bounds follow immediately.  Finally,
$q_Q(x)\le v_0+M^2\max\{\lambda_{\max}(E),0\}$ for every feasible $x$, proving \eqref{eq:posterior}.
\end{proof}

\begin{corollary}[Rank-adaptive certification for $Q\succeq0$]
\label{cor:rank-adaptive}
Let the eigenvalues of $Q\succeq0$ satisfy $\lambda_1(Q)\ge\cdots\ge\lambda_n(Q)\ge0$.  For $0\le r<n$, let $Q_r$ be the rank-$r$ leading spectral truncation, and let $v_r=\max_{x\in\X}x^{\mathsf T}Q_rx$ be computed exactly by rank collapse.  If $\widehat x_r$ is the best true-$Q$ candidate produced by that enumeration, then
\begin{equation}
\label{eq:rank-adaptive}
       L_r:=\widehat x_r^{\mathsf T}Q\widehat x_r
       \le \operatorname{opt}_Q
       \le U_r:=v_r+M^2\lambda_{r+1}(Q).
\end{equation}
Increasing $r$ until $U_r-L_r$ meets a desired tolerance gives a certified rank-adaptive algorithm.  If $r=n$, set $Q_n=Q$ and $U_n=L_n=v_n=\operatorname{opt}_Q$; exact low rank is recovered earlier whenever $\lambda_{r+1}(Q)=0$.
\end{corollary}

The certificate is useful even when candidate enumeration is expensive at the next rank: it separates a lower bound obtained by evaluating the true objective from a spectral upper bound.  It also makes explicit which constants depend on the feasible radius $M$.

\begin{figure}[H]
\centering
\includegraphics[width=0.94\linewidth]{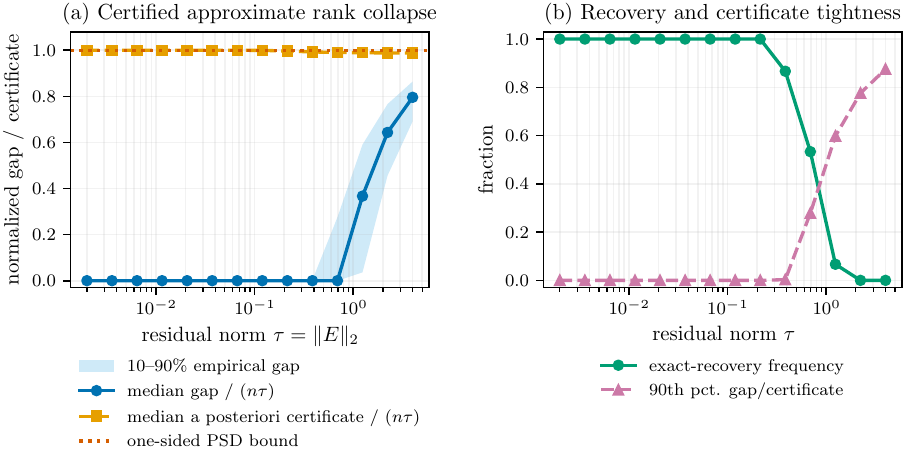}
\caption{Approximately low-rank binary problems with adversarial positive-semidefinite residuals $E=\tau uu^{\mathsf T}$.  Across 420 trials, every observed gap satisfies the one-sided bound $n\tau$ (here $M^2=n$), and the largest gap-to-a-posteriori-certificate ratio is $0.895$.  The experiment deliberately chooses residual directions poorly represented by the rank-two candidate set.}
\label{fig:approximate}
\end{figure}

\section{Curvature boundary and nonclaims}
\label{sec:limits}

The positive-semidefinite orientation in \cref{thm:self-exposure} is essential.

\begin{proposition}[Rank-one negative-semidefinite hardness]
\label{prop:partition}
Given nonzero $a\in\mathbb Z_{>0}^n$, consider
\begin{equation}
\label{eq:partition}
        \max_{x\in\{\pm1\}^n}-(a^{\mathsf T}x)^2.
\end{equation}
The matrix $-aa^{\mathsf T}$ is negative semidefinite of rank one, and the optimum of \eqref{eq:partition} is zero if and only if the associated \textsc{Partition} instance is feasible.  Thus exact low rank alone does not make arbitrary-curvature quadratic maximization polynomial-time unless \textsf{P}=\textsf{NP}.
\end{proposition}

\begin{proof}
The objective is always nonpositive and equals zero precisely when $a^{\mathsf T}x=0$, which assigns the integers to two equal-sum parts.  \textsc{Partition} is NP-complete \citep{gareyjohnson1979}.
\end{proof}

This is weak NP-hardness, matching the weak nature of \textsc{Partition}; it is sufficient to rule out a universal fixed-rank theorem across curvature signs.  The paper also does not claim that the generic arrangement bounds are always tight for top-$k$, matroid, or sparse-PCA families.  Their finer $k$-level and matroidal geometry may yield substantially smaller output-sensitive bounds.  The scattering number is intended to expose precisely where such improvements can enter.

\section{Computational evidence}
\label{sec:experiments}

The experiments provide numerical checks and illustrations of the theory rather than evidence of novelty by themselves.  All code, seeds, raw CSV files, vector figures, and environment metadata accompany the manuscript.  Figure text is rendered through \LaTeX; series are distinguished by markers and line styles as well as color.  The full protocol, including event handling and exhaustive-check details, is in Appendix~\ref{secS:experiments}.

\subsection{Sharp scattering and exact rank-two binary optimization}

For ranks $r=2,3,4$, we generated 132 independent Gaussian matrices over the dimension ranges shown in \cref{fig:binary-scattering}.  The vertices of $B^{\mathsf T}[-1,1]^n$ were computed by convex hull, and all 132 numerical counts agree with
$2\sum_{j=0}^{r-1}\binom{n-1}{j}$.  This is a numerical check of the implementation against the classical generic formula, not a substitute for the zonotope theorem.

The rank-two angular sweep was then compared with exhaustive enumeration on independent Gaussian instances for every even $n$ from 10 through 24; the objective values agreed to numerical precision in every comparison.  On an AMD EPYC 9V74 CPU, median sweep time was $0.115$ seconds at $n=750{,}000$.  Exhaustive search took $2.68$ seconds already at $n=24$.  These timings are implementation-specific, but the contrasting polynomial and exponential trends in \cref{fig:binary-runtime} are the substantive result.

\subsection{Collapse across feasible families}

\Cref{fig:cross-family} compares one rank-two instance from five families.  The count is not presented as a universal average; it is a visual census of the objects actually enumerated by the corresponding oracle partition.  In the displayed instances, 48 binary shadows replace $2^{24}$ points, 72 phase patterns replace $4^{18}$ vectors, 19 top-$8$ subsets and 23 sparse-PCA supports replace $\binom{42}{8}$ possibilities, and 20 spanning trees replace the $9^7$ trees of $K_9$.  These reductions range from $2.4\times10^5$ to $9.5\times10^8$.

\begin{figure}[t]
\centering
\includegraphics[width=0.91\linewidth]{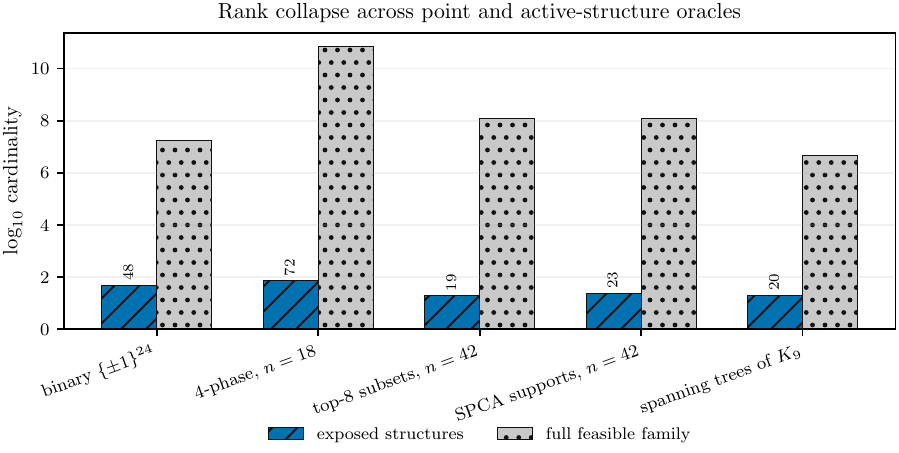}
\caption{A cross-family census of exposed shadows or active structures versus the full feasible family.  Counts are for reproducible rank-two/rank-one instances specified in Appendix~\ref{secS:experiments}; the purpose is to make the combinatorial collapse visible, not to claim these single-instance ratios are distributional estimates.}
\label{fig:cross-family}
\end{figure}

\subsection{Sparse PCA: exactness and finite supports}

For real rank two, active supports were enumerated from the arrangement \eqref{eq:spca-events}.  On 32 independent instances with $n\in\{12,15,18,20\}$, the best enumerated-support eigenvalue matched exhaustive search over all $k$-subsets with zero recorded gap.  In a scaling study with $k\approx0.18n$, the median number of active supports grew from $13.5$ at $n=20$ to 128 at $n=220$, while $\binom{220}{40}$ is astronomically larger.  Median runtime at $n=220$ was $0.373$ seconds.  The left panel of \cref{fig:spca} separately demonstrates the theoretically important fact that finite support collapse does not imply a finite point-oracle image.

\subsection{Approximate rank and directional stability}

The approximate-rank study uses $n=18$ binary instances, a random rank-two core, and a deliberately adversarial residual $E=\tau uu^{\mathsf T}$ whose direction has low correlation with the core candidate set.  There are 30 independent cores and 14 logarithmically spaced residual norms, for 420 trials.  Every gap obeys the $n\tau$ semidefinite-residual theorem bound; the maximum normalized gap is $0.884$, and the maximum ratio of true gap to the computable a posteriori certificate is $0.895$.  Thus the certificate can be close to its one-sided bound under targeted residuals rather than only under benign noise.

Finally, \cref{fig:stability} stress-tests \cref{thm:direction-stability} on a finite rank-two binary shadow set.  Across 10,025 random and adversarial perturbations, the maximum observed ratio $\norm{\widehat y-y^\star}/(2\varepsilon)$ is $0.154$.  The guaranteed exact-recovery threshold is $\varepsilon/\delta<1/2$, while the nearest boundary of the actual normal cone occurs at $2.708$ for this instance, illustrating that the theorem is universal and conservative rather than instance-tight.

\begin{figure}[t]
\centering
\includegraphics[width=0.94\linewidth]{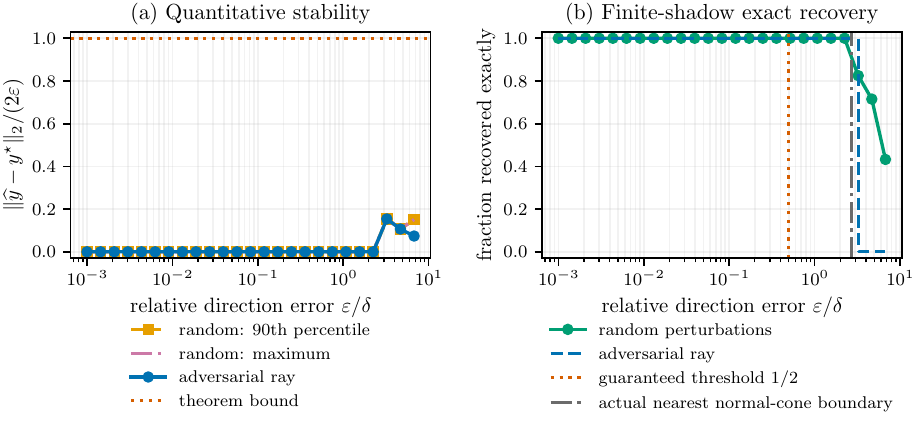}
\caption{Directional stability.  Left: all measured errors satisfy \cref{thm:direction-stability}.  Right: the universal finite-shadow threshold $\varepsilon<\delta/2$ guarantees exact recovery; the particular instance remains stable much farther, until the nearest normal-cone boundary.}
\label{fig:stability}
\end{figure}

\begin{table}[t]
\centering
\caption{Headline reproducibility checks.  Complete records and per-instance values are supplied as CSV files.}
\label{tab:experiments}
\small
\begin{tabularx}{\linewidth}{@{}l X r@{}}
\toprule
Experiment & Check & Result \\
\midrule
Projected cubes & numerical vertex counts match generic formula & 132 / 132 \\
Rank-two binary & sweep objective equals exhaustive objective & all tested $n\le24$ \\
Sparse PCA & active-support optimum equals all-support optimum & 32 / 32 \\
Approximate rank & gap satisfies $M^2\norm E_2$ for $E\succeq0$ & 420 / 420 \\
Directional stability & $\norm{\widehat y-y^\star}\le2\varepsilon$ & 10,025 / 10,025 \\
\bottomrule
\end{tabularx}
\end{table}

\section{Discussion}
\label{sec:discussion}

The rank-collapse principle is deliberately modest in its universal claim and broad in its consequences.  A positive-semidefinite quadratic optimum generates a rank-space normal that exposes its own shadow uniquely.  That statement survives noninjective projections and oracle ties.  What happens next is determined by feasible-set geometry:
\begin{itemize}
\item projected polytopes lead to finite shadow scattering and normal-fan enumeration;
\item unions of tractable pieces lead to active-structure collapse even when point responses are infinite;
\item spectral residuals lead to certified approximate collapse;
\item changing curvature can destroy tractability at rank one.
\end{itemize}

The framework suggests several concrete research directions.  The generic arrangement bound should be replaced by sharper output-sensitive bounds for $k$-levels, graphic matroids, and structured phase alphabets.  Rank-adaptive certification can be combined with lazy projected-fan generation so that rank is increased only when the current spectral certificate requires it.  For sparse PCA and related union-of-subspace problems, the active-structure theorem invites algorithms that enumerate only supports encountered by a normal-fan walk rather than all arrangement chambers.  Finally, the margin \eqref{eq:margin-identity} may support finite-precision and stochastic-oracle analyses beyond the deterministic perturbation bound proved here.

The central lesson is therefore not that ``low rank makes quadratic optimization easy.''  It is sharper:
\begin{quote}
\emph{Low positive-semidefinite rank collapses nonlinear global optimality to low-dimensional self-exposure; the projected or active geometry of the feasible set determines whether that exposure becomes an efficient exact algorithm.}
\end{quote}

\section*{Data and code disclosure}
A general-purpose large language model was used as an assistive tool for drafting and editing, code generation and debugging, and proof exploration.  The accompanying reproducibility package includes the complete Python experiment suite, raw CSV output, environment metadata, and vector PDF figures.  Running \texttt{python code/experiments.py --all} from the package root regenerates the numerical material.  No external data are required.

\clearpage
\appendix
\numberwithin{equation}{section}
\numberwithin{figure}{section}
\numberwithin{table}{section}
\section{Normal-cone interpretation and finite-shadow margins}
\label{secS:normal}

Let $P\subseteq\R^r$ be compact and convex, and let $y^\star$ maximize $\norm y_2^2$ over $P$.  Its normal cone is
\[
N_P(y^\star):=\{z\in\R^r:\inner{z}{y-y^\star}\le0\ \text{for all }y\in P\}.
\]
The self-exposure identity from the main paper implies $y^\star\in N_P(y^\star)$ and, more strongly,
\begin{equation}
\label{eqS:strong-normal}
       \inner{y^\star}{y^\star-y}\ge\frac12\norm{y^\star-y}_2^2
       \qquad (y\in P).
\end{equation}
For a general smooth convex body, $N_P(y^\star)$ can be a ray; unique exposure does not imply a full-dimensional normal cone.  For a polytope, however, unique exposure makes $y^\star$ a vertex and its relative normal cone is full-dimensional in $L=\lin(P-P)$.

\subsection{Distance to a finite normal-fan boundary}

Suppose $\Y\subset\R^r$ is finite and $y^\star$ is its optimal shadow.  For a competitor $y\ne y^\star$, set $d_y=y^\star-y$.  The affine hyperplane in direction space where $y$ ties $y^\star$ is
\[
       \{z:\inner{z}{d_y}=0\}.
\]
Starting at $z=y^\star$, its Euclidean distance is
\begin{equation}
\label{eqS:critical}
       \varepsilon_y^{\mathrm{crit}}
       =\frac{\inner{y^\star}{d_y}}{\norm{d_y}_2}
       \ge\frac12\norm{d_y}_2,
\end{equation}
where the inequality follows from \eqref{eqS:strong-normal}.  Consequently, if
$\delta=\min_{y\ne y^\star}\norm{y-y^\star}_2$, the distance from $y^\star$ to the nearest pairwise tie hyperplane is at least $\delta/2$.  This is the geometric content of the exact-recovery part of the directional-stability theorem.

The bound can be conservative.  It uses only Euclidean separation, whereas the actual boundary depends on the angle between $y^\star$ and each competitor difference.  The direction-stability experiment in the paper reports both $\delta/2$ and the exact minimum of \eqref{eqS:critical}.

\subsection{Duplicate shadows and oracle selections}

Let $\cO_B(z)=\Argmax_{x\in\X}\inner{z}{B^{\mathsf T}x}$.  If two feasible points have the same shadow, no direction can distinguish them.  The main theorem therefore makes two separate assertions:
\begin{enumerate}
\item the optimal \emph{shadow} is a unique linear maximizer;
\item every feasible representative returned at that shadow is quadratically optimal.
\end{enumerate}
It does not assert that a fixed single-valued selection from $\cO_B$ returns every optimal representative.  This distinction removes the tie defect that arises from defining a candidate set through arbitrary deterministic selections.

The zero direction is excluded from meaningful candidate enumeration because $\cO_B(0)=\X$.  If the quadratic optimum is positive, the self-exposing direction $y^\star$ is nonzero.  If it is zero, every shadow vanishes and every feasible point is optimal.

\section{Complete proof of projected edge-direction enumeration}
\label{secS:edge}

Let $P\subseteq\R^n$ be a nonempty polytope, $B\in\R^{n\times r}$, and $P_B=B^{\mathsf T}P$.  Let $L=\lin(P_B-P_B)$ and $d=\dim L$.  A finite set $\cD$ covers the edge directions of $P$ if, for every edge $[u,v]$ of $P$, $v-u=\alpha d_0$ for some $d_0\in\cD$ and nonzero scalar $\alpha$.

For every $d_0\in\cD$, define $g(d_0)=\proj_L(B^{\mathsf T}d_0)$.  Discard $g(d_0)=0$, identify nonzero scalar multiples, and define the proper central hyperplane $H_g=\{z\in L:\inner z g=0\}$.  Denote the resulting arrangement by $\cH$.  The projection is needed because an extraneous covering direction can have a nonzero ambient image but induce the zero functional on $L$.

\begin{lemma}[A chamber maps a maximizing face to one point]
\label{lemS:chamber-singleton}
If $z\in L\setminus\bigcup_{H\in\cH}H$ and
$F=\Argmax_{x\in P}\inner z{B^{\mathsf T}x}$, then $B^{\mathsf T}F$ is a singleton.
\end{lemma}

\begin{proof}
Assume $B^{\mathsf T}F$ is not a singleton.  Then two vertices of $F$ have different images.  The graph of every polytope face is connected, so a path of edges in $F$ connects them.  At least one edge $[u,v]$ on that path satisfies $B^{\mathsf T}(v-u)\ne0$.  Because the entire edge lies in the maximizing face,
\[
       \inner z{B^{\mathsf T}(v-u)}=0.
\]
Its direction is covered by $\cD$.  Since $B^{\mathsf T}(v-u)\in L$, it is a nonzero scalar multiple of the retained projected normal $g(d_0)$.  Hence $z\in H_g$, a contradiction.
\end{proof}

\begin{lemma}[Every projected vertex has a chamber direction]
\label{lemS:vertex-chamber}
For every $y\in\Vtx(P_B)$ there is a full-dimensional chamber $C$ of $\cH$ such that every linear-oracle response at every $z\in C$ projects to $y$.
\end{lemma}

\begin{proof}
The relative normal cone
\[
       N_{P_B}(y):=\{z\in L:\inner z{w-y}\le0\text{ for all }w\in P_B\}
\]
has nonempty interior in $L$ because $y$ is a vertex.  A finite union of proper hyperplanes cannot contain that interior.  Choose
$z\in\operatorname{int}_L N_{P_B}(y)\setminus\bigcup\cH$.  The direction lies in a chamber and uniquely maximizes $\inner z w$ over $P_B$ at $y$.  Any maximizing feasible point in $\X$ therefore projects to $y$.  The same remains true throughout the connected chamber because no projected edge can become tied without crossing an arrangement hyperplane; alternatively, apply \cref{lemS:chamber-singleton} and continuity of the support function.
\end{proof}

\begin{theorem}[Projected edge-direction theorem]
One linear-oracle call over the original feasible set $\X$ from each full-dimensional chamber of $\cH$ produces every vertex of $P_B$, possibly with duplicates.
\end{theorem}

\begin{proof}
By \cref{lemS:chamber-singleton}, each chamber response projects to a single exposed point, hence a projected vertex.  By \cref{lemS:vertex-chamber}, every projected vertex occurs.
\end{proof}

\subsection{Central-arrangement count}

Let $R_{\mathrm c}(N,d)$ be the maximum number of full-dimensional regions cut out in $\R^d$ by $N$ central hyperplanes.  For $N,d\ge1$,
\begin{equation}
\label{eqS:central-count}
       R_{\mathrm c}(N,d)=2\sum_{j=0}^{d-1}\binom{N-1}{j}
\end{equation}
under general position, and this expression is an upper bound without general position \citep{zaslavsky1975arrangements,matousek2002discrete}.  One way to see the recurrence is
\[
       R_{\mathrm c}(N,d)
       \le R_{\mathrm c}(N-1,d)+R_{\mathrm c}(N-1,d-1),
\]
with $R_{\mathrm c}(1,d)=2$ and $R_{\mathrm c}(N,1)=2$.  Adding the $N$th hyperplane splits only regions intersected by it, and the induced arrangement on that hyperplane has dimension at most $d-1$.  Solving the recurrence gives the displayed upper bound; equality holds in general position.  If $N=0$ or $d=0$, there is one region.

For fixed $d$, arrangement construction and cell traversal can be performed in time polynomial in $N$ and the output size \citep{edelsbrunner1987algorithms}.  The manuscript states oracle-call complexity separately from bit complexity because the latter depends on how rational chamber representatives and the underlying linear oracle are implemented.

\section{Complex realification and finite-phase alphabets}
\label{secS:complex}

Let $Q=BB^*$ be Hermitian positive semidefinite, $B\in\C^{n\times r}$, and $x\in\C^n$.  Equip $\C^r$ with the real inner product
\[
       \inner{z}{y}_{\R}:=\operatorname{Re}(z^*y).
\]
Define $\rho:\C^r\to\R^{2r}$ by
$\rho(y)=(\operatorname{Re}y,\operatorname{Im}y)$.  Then
\[
       \inner{z}{y}_{\R}=\rho(z)^{\mathsf T}\rho(y),
       \qquad
       \norm y_2=\norm{\rho(y)}_2.
\]
All self-exposure and stability proofs therefore apply verbatim in real dimension at most $2r$ after replacing transposes by conjugate transposes and ordinary inner products by their real parts.  In particular,
\[
\operatorname{Re}\{(y^\star)^*(y^\star-y)\}
=\frac12\bigl(\norm{y^\star}_2^2-\norm y_2^2+\norm{y^\star-y}_2^2\bigr).
\]

For the alphabet $\Omega_m=\{e^{2\pi\mathrm iq/m}:q=0,\ldots,m-1\}$, the convex hull of one coordinate is a regular $m$-gon.  Its edge directions are the $m$ adjacent phase differences
\[
       e^{2\pi\mathrm i(q+1)/m}-e^{2\pi\mathrm iq/m},
       \qquad q=0,\ldots,m-1.
\]
The product polytope $\conv(\Omega_m^n)$ has edges that vary in one coordinate only, so at most $mn$ coordinate-edge directions suffice.  After projection and realification, each nonzero direction induces a central real hyperplane.  Therefore the number of projected vertices is bounded by
\[
       R_{\mathrm c}(mn,d),\qquad d\le2r.
\]
In complex rank one, directions are points on a circle.  Each coordinate phase quantizer changes at at most $m$ boundary angles, so sorting at most $mn$ events gives an exact angular sweep.  Coincident boundaries are grouped.

\section{Rank-two sweeps, ties, and boundary events}
\label{secS:degeneracy}

\subsection{Binary sweep}

Write row $i$ of $B\in\R^{n\times2}$ as $b_i^{\mathsf T}$ and parameterize a unit direction by $z(\theta)=(\cos\theta,\sin\theta)$.  The binary oracle is
\[
       x_i(\theta)=\operatorname{sign}(b_i^{\mathsf T}z(\theta))
\]
away from event angles.  Because the sign pattern at $\theta+\pi$ is the negative of the sign pattern at $\theta$, one half-circle contains one representative from every antipodal pair.

For generic rows, each event flips one coordinate.  Initialize at a midpoint between consecutive sorted events and form
$y=B^{\mathsf T}x$.  When coordinate $i$ flips,
\[
       y\leftarrow y-2x_i b_i,
       \qquad x_i\leftarrow-x_i.
\]
The objective $\norm y_2^2$ is evaluated after each update.  This is the algorithm used for the runtime experiment.

If several rows have the same event angle, all corresponding signs change when the sweep crosses that angle.  Updating them as a group reaches the adjacent full-dimensional chamber directly.  Intermediate one-at-a-time states would be boundary tie selections and are neither required nor, in general, genuine chamber responses.  Zero rows can be discarded because they do not affect the shadow or objective.  Equivalent correctness can be obtained by an infinitesimal symbolic perturbation followed by deduplication.

\subsection{Cardinality and sparse-PCA support events}

For ordinary top-$k$ selection, score order changes when
$(b_i-b_j)^{\mathsf T}z=0$.  For sparse PCA, the generic theorem in the main paper assumes nonzero rows with $b_i\ne\pm b_j$ for $i\ne j$, so no magnitude comparison is tied identically in every direction.  Under that assumption, magnitude order changes when
\[
       \abs{b_i^{\mathsf T}z}=\abs{b_j^{\mathsf T}z},
\]
which is the union of $(b_i-b_j)^{\mathsf T}z=0$ and $(b_i+b_j)^{\mathsf T}z=0$.  A regular cell then has strict order and hence one top-$k$ support.  At an ordinary boundary, several supports can tie.  Nongeneric permanent ties can instead be represented by tie classes; a symbolic perturbation may be used to select adjacent generic cells, after which all limiting supports must be evaluated rather than an arbitrary single tie-breaking support.  The active-structure theorem shows that it is enough to enumerate supports from the dense set of regular cells: if the self-exposing direction lies on a boundary, a subsequence of adjacent regular cells contains a support whose closure remains active at that direction and contains a global quadratic optimizer.

The implementation used in the rank-two experiments computes all unique event lines in $[0,\pi)$, evaluates the top-$k$ support at every cell midpoint, and deduplicates supports.  It never uses the zero direction.  For each sparse-PCA support $S$, it evaluates $\lambda_{\max}(B_S^{\mathsf T}B_S)$, an $r\times r$ eigenproblem.

\subsection{Boundary cases}

The following conventions make the theorem statements exhaustive.
\begin{itemize}
\item $\X$ is required to be nonempty.  Compactness gives attainment of all continuous objectives and linear support functions.
\item Rank zero is the zero-optimum case: every shadow is zero and every point is optimal.
\item Effective dimension is $d=\dim\lin(P_B-P_B)$, not necessarily the column rank of $B$.
\item Duplicate rows, zero projected edge directions, and coincident hyperplanes reduce the arrangement count; they do not increase it.
\item For top-$k$, $k=0$ and $k=n$ give a singleton feasible set.  Sparse PCA assumes $1\le k\le n$ and unit norm.
\item A projected vertex can have many feasible preimages.  Any preimage returned by a linear oracle has the same quadratic value.
\end{itemize}

\section{Oracle ascent in shadow space}
\label{secS:ascent}

The main paper does not use an iterative method for its exact algorithms, but the self-exposure map gives a useful monotonicity result.  Let $x_t\in\X$, $y_t=B^{\mathsf T}x_t$, and choose
\begin{equation}
\label{eqS:ascent}
       x_{t+1}\in\Argmax_{x\in\X}\inner{y_t}{B^{\mathsf T}x},
       \qquad y_{t+1}=B^{\mathsf T}x_{t+1}.
\end{equation}

\begin{proposition}[Monotone shadow ascent]
For every sequence satisfying \eqref{eqS:ascent},
\[
       \norm{y_{t+1}}_2\ge\norm{y_t}_2.
\]
If $y_{t+1}\ne y_t$, the inequality is strict.  If the shadow set is finite, the shadow sequence terminates after finitely many strict changes at a self-exposed shadow $\bar y$ satisfying
\[
       \bar y\in\Argmax_{y\in\Y_B}\inner{\bar y}{y}.
\]
Every globally optimal shadow is such a fixed shadow, but a fixed shadow need not be globally optimal.
\end{proposition}

\begin{proof}
Oracle optimality gives
$\inner{y_t}{y_{t+1}}\ge\norm{y_t}_2^2$.  Cauchy--Schwarz then gives
$\norm{y_t}_2\norm{y_{t+1}}_2\ge\norm{y_t}_2^2$.  If $y_t\ne0$, divide by $\norm{y_t}_2$; if $y_t=0$, the claim is immediate.  Equality of norms forces equality in Cauchy--Schwarz with positive proportionality, and the displayed oracle inequality then forces $y_{t+1}=y_t$.  A finite shadow set cannot support infinitely many strict norm increases.  The terminal condition is exactly self-exposure.  Global optimal shadows satisfy it by the main theorem.
\end{proof}

A point-valued iteration can move among different feasible representatives with the same terminal shadow.  A self-retaining tie rule terminates the representative sequence as well, but that rule is state-dependent and should not be conflated with a fixed deterministic oracle selection.

\section{Computational protocol and complete numerical summaries}
\label{secS:experiments}

All experiments are generated by
\begin{center}
\texttt{python code/experiments.py --all}.
\end{center}
The master seed is 20260723.  The script writes vector PDF and 250-dpi PNG figures, CSV files, and a JSON environment record.  Convex hulls use SciPy/Qhull without joggling when possible and fall back to Qhull joggling only after a numerical degeneracy; dense linear algebra uses NumPy; the large binary sweep uses a Numba-compiled update kernel when available.  The plots use a colorblind-safe palette, but every series also has a distinct marker or line style.

\begin{table}[ht]
\centering
\caption{Recorded computational environment.}
\label{tabS:environment}
\begin{tabular}{@{}ll@{}}
\toprule
Processor & AMD EPYC 9V74 80-Core Processor \\
Operating platform & Linux 6.12.13, x86\_64, glibc 2.41 \\
Python & 3.13.5 \\
NumPy / pandas / SciPy & 2.3.5 / 2.2.3 / 1.17.0 \\
Matplotlib / Numba & 3.10.8 / 0.65.1 \\
Master seed & 20260723 \\
Run timestamp & 2026-07-27 16:47:22 UTC \\
\bottomrule
\end{tabular}
\end{table}

\subsection{Projected-cube scattering}

For each $r\in\{2,3,4\}$, four independent standard Gaussian matrices were generated at each displayed dimension: $n=5,\ldots,17$ for $r=2$, $n=5,\ldots,15$ for $r=3$, and $n=6,\ldots,14$ for $r=4$.  All $2^n$ shadows were formed and their convex hull computed.  The 132 numerical vertex counts all agree with
$2\sum_{j=0}^{r-1}\binom{n-1}{j}$.  The maximum counts in the tested ranges are 34, 212, and 756 for ranks 2, 3, and 4, respectively.

\subsection{Binary exactness and timing}

For the sweep timings, independent Gaussian $B\in\R^{n\times2}$ matrices were used at
\[
 n\in\{100,300,10^3,3\!\times\!10^3,10^4,3\!\times\!10^4,10^5,3\!\times\!10^5,750{,}000\}.
\]
After one warm-up compilation, five repetitions were recorded.  Exhaustive enumeration was run for even $n=10,\ldots,24$ in bit-generated blocks; the sweep and exhaustive objective values were checked for equality on every exhaustive instance.

\begin{table}[ht]
\centering
\caption{Selected median wall-clock times in seconds.}
\label{tabS:runtime}
\begin{tabular}{@{}r rr@{}}
\toprule
$n$ & Rank-space sweep & Exhaustive search \\
\midrule
$10$ & --- & $8.7\times10^{-5}$ \\
$16$ & --- & $6.62\times10^{-3}$ \\
$20$ & --- & $1.76\times10^{-1}$ \\
$24$ & --- & $2.68$ \\
$10^3$ & $5.7\times10^{-5}$ & --- \\
$10^4$ & $8.01\times10^{-4}$ & --- \\
$10^5$ & $1.14\times10^{-2}$ & --- \\
$750{,}000$ & $1.15\times10^{-1}$ & --- \\
\bottomrule
\end{tabular}
\end{table}

\subsection{Cross-family census}

The cross-family figure uses one reproducible instance per row.  Binary and top-$k$/SPCA rows use standard Gaussian rank-two matrices.  The four-phase row uses a standard complex Gaussian rank-one vector.  The graphic-matroid row assigns rank-two Gaussian vectors to the 36 edges of $K_9$ and invokes Kruskal's maximum-spanning-tree oracle in every regular cell.

\begin{table}[ht]
\centering
\caption{Objects enumerated in the cross-family census.}
\label{tabS:crossfamily}
\small
\begin{tabular}{@{}l r r r@{}}
\toprule
Family & Enumerated & Feasible & Feasible / enumerated \\
\midrule
Binary $\{\pm1\}^{24}$ & 48 & 16,777,216 & 349,525 \\
Four-phase, $n=18$ & 72 & 68,719,476,736 & 954,437,177 \\
Top-$8$ subsets, $n=42$ & 19 & 118,030,185 & 6,212,115 \\
Sparse-PCA supports, $n=42$ & 23 & 118,030,185 & 5,131,747 \\
Spanning trees of $K_9$ & 20 & 4,782,969 & 239,149 \\
\bottomrule
\end{tabular}
\end{table}

\subsection{Sparse PCA}

The exactness study uses eight independent instances for each $(n,k)\in\{(12,3),(15,4),(18,4),(20,5)\}$.  Regular rank-two support cells were enumerated, the maximum eigenvalue was computed on each observed support, and the result was compared with all $\binom{n}{k}$ supports.  All 32 gaps were zero at stored floating-point precision.

The scaling study uses four independent matrices at
\[
n\in\{20,30,45,65,90,125,165,220\},
\qquad k=\max\{3,\operatorname{round}(0.18n)\}.
\]
The support count and time summary is given in \cref{tabS:spca}.

\begin{table}[ht]
\centering
\caption{Rank-two sparse-PCA support enumeration.}
\label{tabS:spca}
\begin{tabular}{@{}rrrrr@{}}
\toprule
$n$ & $k$ & Median supports & Range & Median seconds \\
\midrule
20 & 4 & 13.5 & 11--16 & 0.0012 \\
30 & 5 & 15 & 14--15 & 0.0022 \\
45 & 8 & 24 & 22--28 & 0.0054 \\
65 & 12 & 37 & 35--40 & 0.0140 \\
90 & 16 & 49 & 47--60 & 0.0276 \\
125 & 22 & 71 & 67--78 & 0.0577 \\
165 & 30 & 94 & 87--107 & 0.173 \\
220 & 40 & 128 & 122--135 & 0.373 \\
\bottomrule
\end{tabular}
\end{table}

\subsection{Approximate rank}

For each of 30 independent $n=18$ binary instances, a Gaussian rank-two core $B$ was scaled by $1/\sqrt n$.  From 3,000 randomly sampled sign vectors, the code selected a vector with minimum maximum absolute correlation with the core candidate set, normalized it to $u=x_{\mathrm{bad}}/\sqrt n$, and formed the positive-semidefinite residual $E=\tau uu^{\mathsf T}$.  Fourteen $\tau$ values span $[0.002,4]$ geometrically.  Because $\norm E_2=\tau$ and every sign vector has norm $\sqrt n$, the theorem bound is $n\tau$.

Across all 420 trials, the maximum true gap divided by $n\tau$ is 0.8836.  The maximum true gap divided by the computable a posteriori certificate is 0.8947.  The core candidate set remains exactly optimal in 74.8\% of all trials, with recovery naturally decreasing as $\tau$ grows.

\subsection{Directional stability}

A single reproducible $n=35$ rank-two binary instance was used.  All regular binary shadows were enumerated, the maximum-norm shadow $y^\star$ identified, and
$\delta=\min_{y\ne y^\star}\norm{y-y^\star}_2$ computed.  At each of 25 logarithmically spaced relative radii, 400 uniformly random directions and one adversarial direction aimed at the nearest pairwise normal-cone boundary were tested.  This produces 10,025 records.

The largest observed ratio $\norm{\widehat y-y^\star}_2/(2\varepsilon)$ is 0.1541.  The exact nearest boundary lies at $\varepsilon/\delta=2.7076$, whereas the theorem guarantees exact recovery for $\varepsilon/\delta<0.5$.  Random exact recovery remains 100\% through the tested radius 2.2475 and declines thereafter.  These data demonstrate validity of the universal bound while quantifying its instance-specific slack.

\subsection{File inventory and reproducibility}

The package contains:
\begin{itemize}
\item \texttt{code/experiments.py}: all algorithms, simulations, checks, and plots;
\item \texttt{results/*.csv}: raw records for each experiment;
\item \texttt{results/metadata.json}: seed and environment;
\item \texttt{figures/*.pdf}: vector figures used in the paper;
\item \texttt{arxiv.tex}, \texttt{main\_body.tex}, \texttt{appendices.tex}, and \texttt{refs.bib}: manuscript sources.
\end{itemize}
The script raises an exception if an exact sparse-PCA check fails.  Binary exactness checks are performed during the timing experiment.  No values in the paper are manually inserted into the CSV files.

\bibliographystyle{abbrvnat}
\bibliography{refs}

\end{document}